\documentclass[a4paper,11pt]{article}

\usepackage{amsmath,amsthm,amsfonts,amssymb,amscd}
\usepackage{verbatim}
\usepackage{color}
\usepackage{hyperref,cite,mdwlist}
\usepackage{manyfoot}
\usepackage[english]{babel}
\usepackage{cite}
\usepackage{latexsym}
\usepackage{amsopn}
\usepackage[all]{xy}
\usepackage{mathrsfs}
\usepackage{float}
\usepackage[active]{srcltx}
\usepackage{fullpage}
\usepackage{tikz}
\usepackage{dsfont}
\usetikzlibrary{matrix}

\usepackage[bitstream-charter]{mathdesign}

\newtheorem{inizio}{Lemma}[section]
\newtheorem{theorem}[inizio]{Theorem}
\newtheorem{corollary}[inizio]{Corollary}
\newtheorem{proposition}[inizio]{Proposition}
\newtheorem{lemma}[inizio]{Lemma}

\newtheorem{conjecture}[inizio]{Conjecture}

\newtheorem*{theorem*}{Theorem}
\newtheorem*{theoremA}{Theorem A}
\newtheorem*{theoremB}{Theorem B}
\newtheorem*{theoremC}{Theorem C}
\newtheorem*{conjectureD}{Conjecture D}

\theoremstyle{definition}
\newtheorem{definition}[inizio]{Definition}
\newtheorem{example}[inizio]{Example}
\newtheorem{remark}[inizio]{Remark}

\renewcommand{\to}{\longrightarrow}

\title{On factoriality of threefolds with isolated
singularities}

\author{Francesco Polizzi, Antonio Rapagnetta and Pietro Sabatino}

\date{}

\begin{document}

\maketitle

\begin{abstract}
We investigate the existence of complete intersection threefolds $X
\subset \mathbb{P}^n$ with only isolated, ordinary multiple points
and we provide some sufficient conditions for their factoriality.
\end{abstract}

\Footnotetext{{}}{\textit{2010 Mathematics Subject Classification}:
14J30, 14J70}

\Footnotetext{{}} {\textit{Keywords}: complete intersection
threefold, isolated singularity, factoriality.}


\section{Introduction}
Grothendieck-Lefschetz's Theorem (\cite[Expos$\acute{\textrm{e}}$
XII, Corollaire 3.7]{SGA2}, \cite[Chapter IV, Corollary
3.3]{Ha70}, \cite[Corollary 2.3.4]{bs:av}) says that if $X$ is an
effective, ample divisor of a smooth variety $Y$, defined over a
field of characteristic $0$, then the restriction map of Picard
groups
\begin{equation*}
\textrm{Pic}\,Y \to \textrm{Pic}\,X
\end{equation*}
is injective if $\dim X \geq 2$ and is an isomorphism if $\dim X
\geq 3$. One might ask what happens, with the same hypotheses, to
the restriction map $\textrm{CH}^1(Y)\rightarrow \textrm{CH}^1(X)$
between rational equivalence classes of codimension $1$
subvarieties. Under some mild assumptions on the singularities of
$X$ (e.g. if $X$ is normal) this is equivalent to asking whether
or not the conclusions of Grothendieck-Lefschetz's Theorem for
Picard groups remain true for the restriction map
\begin{equation} \label{eq:cl}
\textrm{Cl}\,Y \to \textrm{Cl}\,X,
\end{equation}
where, as usual, $\textrm{Cl}\,X$ denotes the \emph{class group} of
$X$, namely the group of linear equivalence classes of Weil
divisors. When $X$ is smooth there is nothing new to say, since the
groups $\textrm{Pic}\,X$ and $\textrm{Cl}\,X$ are isomorphic;
however, when $X$ is singular the problem becomes a delicate one.

We will restrict ourselves to the case where $Y \subset
 \mathbb{P}^n$ $(n \geq 4)$ is a smooth, complete intersection fourfold
 and $X \subset Y$ is a threefold with isolated singularities. Since $X$ is
projectively normal and nonsingular in codimension $1$, the map
\eqref{eq:cl} is an isomorphism precisely when
$\textrm{Pic}\,X=\textrm{Cl}\,X = \mathbb{Z}$, generated by the
class of $\mathcal{O}_X(1)$. This is in turn equivalent to the
fact that the homogeneous coordinate ring of $X$ is a UFD, or that
any hypersurface in $X$ is the complete intersection of $X$ with a
hypersurface of $\mathbb{P}^n$. In this case we say that $X$ is
\emph{factorial}.

 In the recent years, the study of factoriality of threefolds in
 $\mathbb{P}^4$ having only ordinary double points (``nodes") has attracted the
 attention of several authors. In particular the following result was
 conjectured, and proven in a weaker form, by Ciliberto and Di Gennaro
 (\cite{CdG04}). The proof of the general case is due to Cheltsov (\cite{Ch10},
 \cite{ch-sb}).
\begin{theorem*}
 Let $X \subset \mathbb{P}^4$ be a nodal threefold of degree $d$, and set
 $k=|\emph{Sing}(X)|$.
 If $k < (d-1)^2$ then $X$ is factorial, whereas when $k=(d-1)^2$ then $X$
 is factorial if
 and only if the nodes are not contained in a plane. Moreover, if
 the nodes are contained in  a plane $\pi$, then necessarily $\pi \subset X$
 and this explains the lack of factoriality in this case.
\end{theorem*}

In the present paper we deal with the situation where the
singularities involved are not necessarily nodes but, more
generally, ordinary $m$-ple points with $m \geq 2$. Here ``ordinary"
means that the corresponding tangent cone is a cone over a smooth
surface in $\mathbb{P}^3$. We first prove the following
Lefschetz-type result, see Theorem \ref{theorem:principal}.

\begin{theoremA}
Let $Y \subseteq \mathbb{P}^n$ be a smooth, complete intersection
fourfold and let $X \subseteq Y$ be a reduced and irreducible
threefold which is the intersection of $Y$ with a hypersurface of
  $\mathbb{P}^n$.
  Suppose that the singular locus $\Sigma$ of  $X$ consists of isolated,
  ordinary
  multiple points and denote by $\widetilde{X}\subset \widetilde{Y}$ the
  strict transform of
  $X$ in the blowing-up $\widetilde{Y}:=\mathrm{Bl}_{\Sigma}Y$ of
  $Y$ at $\Sigma$.
If $\widetilde{X}$ is ample in $\widetilde{Y}$, then $X$ is
factorial.
\end{theoremA}

The first consequence is that if $X$  has ``few" singularities,
which are all ordinary points, then $X$ is factorial. More
precisely, we have the following result, see Theorem
\ref{theorem:few-points}.

\begin{theoremB}
Let $Y \subset \mathbb{P}^n$ be a smooth, complete intersection
fourfold and $X \subset Y$  be a reduced, irreducible threefold,
which is complete intersection of $Y$ with a hypersurface of degree
$d$. Assume that the singular locus of $X$ consists of $k$ ordinary
multiple points $p_1,  \ldots, p_k$ of multiplicity $m_1, \ldots,
m_k$. If
\begin{equation} \label{eq:linear-bound-1-gen-int}
\sum_{i=1}^k m_i < d
\end{equation}
then $X$ is factorial.
\end{theoremB}

We also give a different proof of Theorem B in the case $X \subset
\mathbb{P}^4$ and $k=1$ (see the Appendix), because we find it of
independent interest.

Theorem B provides the first factoriality criterion for complete
intersection threefolds in $\mathbb{P}^n$ with ordinary
singularities. Previously, only results for \emph{nodal}
threefolds in $\mathbb{P}^4$ and $\mathbb{P}^5$ were known
(\cite{Ch10}, \cite{kosta:ci}).

When $X \subset \mathbb{P}^4$ and the inequality
\eqref{eq:linear-bound-1-gen-int} is not satisfied, we can still
give a factoriality criterion provided that the singularities of $X$
are in general position and they all have the same multiplicity. In
fact, using Theorem A together with a result of Ballico (Theorem
\ref{theorem:ballico}), we deduce the following result, see Theorem
\ref{theorem:worstthannodes}.

\begin{theoremC}
Let $\Sigma:=\{p_1,\ldots, p_k \}$ be a set of $k$ distinct,
  general points in $\mathbb{P}^4$ and let $d, \, m$ be
  positive integers with $d \geq m$.
\begin{itemize}
\item[$\boldsymbol{(i)}$] If
  \begin{equation}
    \bigg \lfloor \frac{d+5}{m+4} \bigg \rfloor^4 > k
\end{equation}
  then there exists a hypersurface $X \subset \mathbb{P}^4$ of degree $d$,
  with
  $k$ ordinary $m$-ple points at $p_1,\ldots,p_k$ and no other
  singularities. \\
\item[$\boldsymbol{(ii)}$]  If the stronger condition
  \begin{equation}
   \min \bigg\{ \bigg \lfloor \frac{d+5}{m+4}  \bigg \rfloor^4, \, \bigg
   \lfloor \frac{d}{m}
   \bigg \rfloor^4 \bigg\} > k
\end{equation}
  holds, then any  hypersurface $X$ as in part $\boldsymbol{(i)}$
  is factorial.
\end{itemize}
\end{theoremC}

Using Theorem C one can easily provide new examples of singular,
factorial projective varieties.

In the last part of the paper we construct some non-factorial
threefolds $X \subset \mathbb{P}^4$ of degree $d$ with only
 $k$ isolated, ordinary $m$-ple points as singularities. In all these examples
 the equality
$k(m-1)^2=(d-1)^2$ is satisfied. On the other hand, in
\cite{Sab04} it is proven that if the singular locus of $X $
consists of $k_2$ ordinary double points and $k_3$ ordinary triple
points and if $k_2 + 4 k_3 < (d -1)^2$, then any smooth surface
contained in $X$ is a complete intersection in $X$. Motivated by
this fact, we make the following conjecture, which generalizes the
results of Ciliberto, Di Gennaro and Cheltsov.

\begin{conjectureD}
Let $X \subset \mathbb{P}^4$ be a hypersurface of degree $d$ whose
singular locus consists of $k$ ordinary multiple points $p_1,
\ldots, p_k$ of multiplicity $m_1, \ldots, m_k$. If
\begin{equation*}
\sum_{i=1}^k (m_i-1)^2 < (d-1)^2
\end{equation*}
then $X$ is factorial.
\end{conjectureD}
We hope to come back on this problem in a sequel to this paper.

Let us now explain how this work is organized. In Section
\ref{section:preliminaries} we fix notation and terminology and we
collect some preliminary results that are needed in the sequel of
the paper. In Section \ref{section:factoriality} we discuss the
notion of factoriality of projective varieties and its relations
with the close concepts of local factoriality and
$\mathbb{Q}$-factoriality, providing several examples and
counterexamples.

 In Section \ref{section:lefschetz-type}
 we prove Theorem A, whereas in Section \ref{section:applications}
 we present some of its consequences, including Theorem B and Theorem C.
 Finally, in Section \ref{section:non-factorial} we describe our
 non-factorial examples and we state Conjecture D.


\bigskip

\textbf{Acknowledgments.} F. Polizzi was supported by Progetto MIUR
di Rilevante Interesse Nazionale \emph{Geometria delle
Variet$\grave{a}$ Algebriche e loro Spazi di Moduli} and by
GNSAGA-INdAM.

A. Rapagnetta was supported by  Progetto FIRB \emph{Spazi di Moduli
e Applicazioni}.

The authors are grateful to E. Di Gennaro and F. Zucconi for useful
discussions and suggestions and to J. Starr and H. Dao for answering
some of their questions on the interactive mathematics website
MathOverflow (\verb|www.mathoverflow.net|).

\bigskip

\textbf{Notation and conventions.} We work over the field
$\mathbb{C}$ of complex numbers.

If $X$ is a projective variety and $D_1, \, D_2$ are divisors on
$X$, we write $D_1 \equiv_{\textrm{hom}} D_2$ for homological
equivalence and $D_1 \equiv_{\textrm{lin}} D_2$ for linear
equivalence.

The group of linear equivalence classes of Weil divisors on $X$ is
denoted by $\textrm{Cl}\, X$, whereas the group of linear
equivalence classes of Cartier divisors is denoted by $\textrm{Pic}
\, X$.

We write $\textrm{Sing}\, X$ for the singular locus of $X$ and
$b_k(X)$ for its $k$-th Betti number.

\section{Preliminaries} \label{section:preliminaries}

In this section we collect, for the reader's convenience, some
preliminary results that are used in the sequel of the paper. We
start by stating some versions of Lefschetz's theorem on
hyperplane sections, both for cohomology groups and for Picard
groups.

\begin{theorem} \label{teo:lefschetz}
Let $Y$ be a smooth, projective variety of dimension $n$ and $X=X_1
\cap \ldots \cap X_e \subseteq Y$ a smooth complete intersection of
effective, ample divisors on $Y$. Then the restriction map
\begin{equation*}
H^i(Y, \, \mathbb{Z}) \to H^i(X, \, \mathbb{Z})
\end{equation*}
is an isomorphism for $i \leq \dim X-1= n-e-1$, and is injective
with torsion-free cokernel for $i= \dim X=n-e$.
\end{theorem}

\begin{proof}
See \cite[Remark 3.1.32]{lazarsfeld:p1}.
\end{proof}

\begin{theorem} \label{theorem:dimca}
Let $X \subset \mathbb{P}^n$ be a complete intersection which has
only isolated singular points. Then the restriction map
\begin{equation*}
H^i(\mathbb{P}^n, \, \mathbb{C}) \to H^i(X, \, \mathbb{C})
\end{equation*}
is an isomorphism for $\dim X +2 \leq i \leq 2 \dim X$.
\end{theorem}

\begin{proof}
See \cite[Theorem 2.11 p. 144]{dimca:sing}.
\end{proof}

\begin{theorem}\label{theorem:gr-lef}
Let $X \subset \mathbb{P}^n$ be a reduced complete intersection with
$\dim X \geq 3$. Then the restriction map
\begin{equation*}
\emph{Pic} \,\mathbb{P}^n \to \emph{Pic}\,X
\end{equation*}
is an isomorphism. In particular $\emph{Pic}\,X = \mathbb{Z}$,
generated by $\mathcal{O}_X(1)$.
\end{theorem}

\begin{proof}
See \cite[Chapter IV, Corollary 3.2]{Ha70} or
\cite[Expos$\acute{\textrm{e}}$ XII, Corollaire 3.7]{SGA2}.
\end{proof}

Theorem \ref{theorem:gr-lef} admits the following generalization.
\begin{theorem} \label{theorem:gr-lef-gen}
Let $Y$ be a smooth, projective variety and $X \subset Y$ a reduced,
effective ample divisor. Then the restriction map
\begin{equation*}
\emph{Pic}\,Y \to \emph{Pic}\,X
\end{equation*}
is an isomorphism if $\dim X \geq 3$ and is injective with torsion
free cokernel if $\dim X=2$.
\end{theorem}

\begin{proof}
See \cite[Corollary 2.3.4]{bs:av} and \cite[Example
3.1.25]{lazarsfeld:p1}.
\end{proof}

We will also need the following ampleness criterion for the
blow-up of $\mathbb{P}^n$ at a finite number of general points.
\begin{theorem}  \label{theorem:ballico}
  Fix integers $n, \, k, \,d$ with $n\ge 2, \ d\ge 2$ and $k>0;$ if $n=2$
  assume
  $d\ge 3$. Let $p_1,\ldots,p_k \in \mathbb{P}^n$ be general points,
  denote by $\pi \colon \widetilde{\mathbb{P}}^n \to \mathbb{P}^n$ the blow-up of
  $\mathbb{P}^n$ at $p_1,\ldots, p_k$, with exceptional divisors
  $E_1, \ldots, E_k$, and set $H:= \pi^*\mathcal{O}_{\mathbb{P}^n}
  (1)$.
  Then the divisor
  \begin{equation*}
    L:=dH- \sum_{i=1}^kE_i
  \end{equation*}
  is ample in $\widetilde{\mathbb{P}}^n$ if and only if $L^n>0$, or equivalently if and only
  if $d^n> k$.
\end{theorem}

\begin{proof}
See \cite{ballico:ample}.
\end{proof}

\begin{corollary} \label{corollary:amplenessbound}
  Same notation as in Theorem \emph{\ref{theorem:ballico}}. If $a, \, b$
  are positive
  integers such that $\big \lfloor \frac{a}{b} \big \rfloor^n > k$, then
  the divisor
  \begin{equation*}
   L:= aH- b \sum_{i=1}^k E_i
  \end{equation*}
  is ample.
\end{corollary}

\begin{proof}
  Write $a=\big \lfloor \frac{a}{b} \big \rfloor \cdot b + r$, where $r$
  is an integer such that $0\le r < b$. Then we have
  \begin{equation*}
    \begin{split}
     L= aH- b\sum_{i=1}^k E_i & = \bigg(\bigg \lfloor \frac{a}{b} \bigg \rfloor
      \cdot b + r \bigg)H -b \sum_{i=1}^k E_i
      \\
      & =  b \bigg( \bigg \lfloor \frac{a}{b} \bigg \rfloor H - \sum_{i=1}^k
      E_i  \bigg)+ r
      H.
    \end{split}
  \end{equation*}
  The first summand is ample by Theorem \ref{theorem:ballico} and
  the second is nef, so $L$  is ample by \cite[Corollary 1.4.10
  p. 53]{lazarsfeld:p1}.
\end{proof}

\section{Factoriality} \label{section:factoriality}

\begin{definition}
An integral domain $A$ is called a \emph{unique factorization
domain} (abbreviated to UFD) if any element, which is neither $0$
nor a unit, factors uniquely (up to order and units) into a
product of irreducible elements.
\end{definition}

By \cite[Theorem 20.1]{Mat89}, a noetherian integral domain is a
UFD if and only if every height $1$ prime ideal is principal.
Moreover a noetherian, integrally closed domain $A$ is a UFD if
and only if $\textrm{Cl}\,A=0$, where $\textrm{Cl}\,A$ denotes the
\emph{divisor class group} of $A$, namely the group of divisorial
fractional ideals modulo the subgroup of principal fractional
ideals, see \cite[p. 165]{Mat89}.

\begin{proposition} \label{proposition:factoriality-localization}
If $A$ is a noetherian \emph{UFD} and $S \subset A$ is a
multiplicative part, then $S^{-1}A$ is a \emph{UFD}. In
particular, if $\mathfrak{p} \subset A$ is a prime ideal then the
local ring $A_{\mathfrak{p}}$ is a \emph{UFD}.
\end{proposition}
\begin{proof}
Take a height $1$ prime ideal $P \subset S^{-1}A$; then there
exists a prime ideal $I$ of $A$ such that $P=S^{-1}I$.
Localization does not change height, so $I$ has height $1$ in $A$
and since $A$ is a UFD we conclude that $I$ is principal, say
$I=\langle a \rangle$. Then $P=S^{-1}\langle a \rangle=\langle
\frac{a}{1} \rangle$, so $P$ is principal and this concludes the
proof.
\end{proof}





Let $(X, \, \mathcal{O}_X)=(\textrm{Spec} \, A, \, \widetilde{A})$
be the affine scheme associated with a commutative ring $A$; then
we write $\textrm{Pic} \, A$ in place of $\textrm{Pic} \, X$.
Assuming that $A$ is a noetherian domain with only a finite number
of maximal ideals $\mathfrak{m}_1, \ldots, \mathfrak{m}_k$ such
that $A_{\mathfrak{m}_i}$ is not a UFD, there is a short exact
sequence
\begin{equation} \label{seq:Jaffe-local}
0 \to \textrm{Pic}\, A \to \textrm{Cl} \, A \to
\textrm{Cl}(S^{-1}A) \to 0,
\end{equation}
where $S=A-\bigcup_{i=1}^k \mathfrak{m}_i$. See \cite[Chapter
V]{Fo73} for more details.

\begin{proposition} \label{prop:punctured-spectrum}
Let $(A, \, \mathfrak{m})$ be a noetherian, normal local ring with
$\dim A \geq 2$ and set $U:=\emph{Spec}\,A - \mathfrak{m}$. Then
\begin{itemize}
\item[$\boldsymbol{(i)}$] there is a monomorphism $\emph{Pic}\,U \to
\emph{Cl}\, A;$ \item[$\boldsymbol{(ii)}$] if $A_{\mathfrak{p}}$ is
a \emph{UFD} for all $\mathfrak{p} \in U$, then $\emph{Pic}\,U \to
\emph{Cl}\, A$ is an isomorphism.
\end{itemize}
\end{proposition}
\begin{proof}
See \cite[Proposition 18.10]{Fo73}.
\end{proof}

\begin{definition} \label{definition:factorial}
Let $X \subset \mathbb{P}^n$ be a projective variety. We say that
$X$ is \emph{factorial} if its homogeneous coordinate ring
$S(X)=\mathbb{C}[x_0, \ldots, x_n]/I_X$ is a UFD.
\end{definition}

\begin{proposition} \label{proposition:fact-nonsing-codim-1}
If $X$ is projectively normal and nonsingular in codimension $1$,
then $X$ is factorial if and only if the group $\emph{Cl} \,X$ is
isomorphic to $\mathbb{Z}$, generated by $\mathcal{O}_X(1)$.
Equivalently, $X$ is factorial if and only if the restriction map
\begin{equation*}
\emph{Cl} \,\mathbb{P}^n \to \emph{Cl}\,X
\end{equation*}
is an isomorphism.
\end{proposition}
\begin{proof}
See \cite[Exercise 6.3 (c) p. 147]{Ha77}.
\end{proof}

\begin{remark} \label{rem:factoriality}
Using Proposition \ref{proposition:fact-nonsing-codim-1} and
Theorem \ref{theorem:gr-lef}, one  sees that if $X$ is a complete
intersection, nonsingular in codimension $1$ and such that $\dim X
\geq 3$, then $X$ is factorial if and only if $\textrm{Pic}\, X =
\textrm{Cl}\, X = \mathbb{Z}$, generated by $\mathcal{O}_X(1)$.
\end{remark}

\begin{proposition} \label{proposition:samuel}
Let $X \subset \mathbb{P}^n$ be a complete intersection such that
$\dim (\emph{Sing}\,X) < \dim X -3$. Then $X$ is factorial.
\end{proposition}
\begin{proof}
This follows from Grothendieck's proof of Samuel's conjecture, see
\cite[Exp. XI, Corollaire 3.14]{SGA2} and \cite[p. 168]{Mat89}.
\end{proof}
Notice that Proposition \ref{proposition:samuel} implies that any
complete intersection of dimension at least $4$ and with only
isolated singularities is necessarily factorial. This explains why
in the sequel we will restrict ourselves to the case where $X$ is
a threefold.

Let us provide now a couple of examples showing that factoriality
is a subtle property, which cannot be detected by merely looking
at the type of singularities of $X$.

\begin{example} \label{example:fact-1}
Take a hypersurface $X \subset \mathbb{P}^4$ of degree $d$ with a
unique ordinary double point and no other singularities. If $d
\geq 3$ then $X$ is factorial, see \cite{Ch10}. By contrast, if
$d=2$ then $X$ is a cone over a smooth quadric surface in
$\mathbb{P}^3$, which is not factorial because any plane contained
in $X$ is a Weil divisor which is not Cartier. Notice that, since
all ordinary double points are analytically isomorphic, it is
impossible to tell locally analytically the difference between the
two cases $d \geq 3$ and $d=2$, see also \cite[p. 160-161]{De01}.
\end{example}
\begin{example} \label{example:fact-2}
Take a hypersurface $X \subset \mathbb{P}^4$ of degree $d$ with
exactly $(d-1)^2$ ordinary double points and no other singularities.
Then $X$ is factorial if and only if the nodes are not contained in
a plane, see \cite{ch-sb}. Up to change of coordinates, the fact
that the nodes are contained in a plane is equivalent to the fact
that the equation of $X$ can be written as $x_0F+x_1G=0$, hence the
whole plane $\{x_0=x_1=0 \}$ is contained in $X$. Such a plane is a
Weil divisor which is not Cartier and this explains the lack of
factoriality in this case.
\end{example}

\begin{definition} \label{definition:loc-factorial}
We say that $X \subset \mathbb{P}^n$ is \emph{locally factorial}
if the local ring $\mathcal{O}_{X,p}$ is a UFD for any $p \in X$.
We say that $X \subset \mathbb{P}^n$ is \emph{locally analytically
factorial} if the complete local ring
$\widehat{\mathcal{O}}_{X,p}$ is a UFD for any $p \in X$.
\end{definition}
Since every regular local ring is a UFD (\cite[Theorem 20.3]{Mat89})
and the completion of a regular local ring is again regular
(\cite[Exercise 19.1 p. 488]{Eis94}), it suffices to check the UFD
property only at the points $p \in \textrm{Sing}\,X$. An immediate
consequence of Proposition
\ref{proposition:factoriality-localization} is that if an
irreducible projective variety $X \subset \mathbb{P}^n$ is
factorial, then it is locally factorial. By using Remark
\ref{rem:factoriality} and \cite[Chapter II, Proposition 6.11]{Ha77}
one can prove the following more precise result.

\begin{proposition} \label{proposition:fac=locally-fac} Let $X
\subset \mathbb{P}^n$ be a complete intersection, non singular in
codimension $1$ and such that $\dim X \geq 3$. Then $X$ is
factorial if and only if $X$ is locally factorial.
\end{proposition}
The condition $\dim X \geq 3$ in the statement of Proposition
\ref{proposition:fac=locally-fac} is an essential one. In fact,
take any smooth surface $V \subset \mathbb{P}^3$ of degree at
least $2$ and containing a line. Then $V$ is not factorial (since
the line is a divisor which is not an integer multiple of the
hyperplane section), but it is locally factorial because it is
nonsingular.

By Mori's theorem (\cite[Corollary 6.12]{Fo73}) there is a
monomorphism $\textrm{Cl}\, \mathcal{O}_{X,p} \to \textrm{Cl}\,
\widehat{\mathcal{O}}_{X,p}$; this implies that if $X$ is locally
analytically factorial then $X$ is locally factorial. The converse
is in general not true, as shown by the following examples.

\begin{example} \label{ex:anlocfac-vs-locfac}
Take a factorial hypersurface $X \subset \mathbb{P}^4$ with a node
$p$. Then $X$ is locally factorial (Proposition
\ref{proposition:fac=locally-fac}), so the ring
$\mathcal{O}_{X,p}$ is a UFD. However its completion
$\widehat{\mathcal{O}}_{X,p}$ is not a UFD, since it is isomorphic
to $\mathbb{C}[[x, \, y, \, z, \, w]]/(xy-zw)$ and the equality
$xy=zw$ is a product of irreducibles in two different ways.
\end{example}

\begin{example}[\cite{Lip75}] \label{example:cone-anlocfac}
Let $X$ be a cone over a smooth, projectively normal variety $V
\subset \mathbb{P}^{n-1}$. Then $X$ is factorial if and only if
$\textrm{Pic}\,V=\mathbb{Z}$, generated by $\mathcal{O}_V(1)$.
Moreover, $X$ is locally analytically factorial if and only if it is
factorial and, in addition,
\begin{equation}\label{eq:cond-locanfac}
H^1(V, \, \mathcal{O}_V(k))=0 \quad \textrm{for all } k >0.
\end{equation}
Condition \eqref{eq:cond-locanfac} is satisfied e.g. if $V$ is a
complete intersection and $\dim V \geq 2$.
\end{example}

\begin{definition}
A projective variety $X$ is called $\mathbb{Q}$-\emph{factorial}
if every Weil divisor on $X$ has an integer multiple which is a
Cartier divisor.
\end{definition}
Setting $G(X):=\textrm{Cl}\, X / \textrm{Pic}\, X$, we see that $X$
is factorial if and only if $G(X)=0$ and $X$ is
$\mathbb{Q}$-factorial if and only if $G(X)$ is a torsion group. In
particular, if $X$ is factorial then it is $\mathbb{Q}$-factorial.
For threefolds which are a complete intersection in a smooth ambient
space the converse also holds.

\begin{proposition} \label{prop:Q-fact}
Let $Y \subset \mathbb{P}^n$ be a smooth fourfold and let $X
\subset Y$ be a reduced and irreducible threefold with isolated
singularities, which is the intersection of $Y$ with a
hypersurface of $\mathbb{P}^n$. Then $X$ is factorial if and only
if $X$ is $\mathbb{Q}$-factorial.
\end{proposition}
\begin{proof}
Let us assume that $G(X)$ is a torsion group; we want to prove that
$G(X)=0$. According to \cite[\S 1]{BiSt81} and \cite[Section
2]{HaPo13}, from \eqref{seq:Jaffe-local} one obtains the so-called
\emph{Jaffe's exact sequence}, namely
\begin{equation} \label{seq:Jaffe-global}
0 \to \textrm{Pic}\, X \to \textrm{Cl}\, X \to \bigoplus_{p \in
\textrm{Sing}\, X} \textrm{Cl}\, \mathcal{O}_{X, p},
\end{equation}
so we have a monomorphism
\begin{equation} \label{equation:G(X)}
G(X) \to \bigoplus_{p \in \textrm{Sing}\, X} \textrm{Cl}\,
\mathcal{O}_{X, p}.
\end{equation}
Let $\mathfrak{m}_p$ be the maximal ideal of $\mathcal{O}_{X, p}$
and set $U_p := \textrm{Spec}\,\mathcal{O}_{X, p} - \mathfrak{m}_p$;
then there is an isomorphism $\textrm{Pic}\, U_p \to \textrm{Cl}\,
\mathcal{O}_{X,p}$, see Proposition \ref{prop:punctured-spectrum}.
On the other hand, $\textrm{Pic}\, U_p$ is torsion-free by
\cite{Rob76}, see also \cite{Dao12}. It follows that $\textrm{Cl}\,
\mathcal{O}_{X,p}$ is torsion-free, so \eqref{equation:G(X)} yields
$G(X)=0$.
\end{proof}

\begin{remark}
It is possible to give a different proof of Proposition
\ref{prop:Q-fact} using Theorem \ref{theorem:gr-lef-gen}, see
\cite{Sab09}. It is also true that  any $\mathbb{Q}$-factorial
Gorenstein threefold with terminal singularities is factorial
(\cite{Cut88}); however notice that an ordinary threefold
singularity of multiplicity $m$ is terminal if and only if $m \leq
3$ (\cite[p. 351]{Re87}). For the relevance of the concept of
$\mathbb{Q}$-factoriality in the setting of birational geometry, we
refer the reader to \cite{Mel04}.
\end{remark}

\begin{remark}
If $X$ is not a complete intersection, then Proposition
\ref{prop:Q-fact} is in general not true. For instance, the cone
over the Veronese surface $V \subset \mathbb{P}^5$ is
$\mathbb{Q}$-factorial but not factorial, see \cite[p. 20]{bs:av}.
\end{remark}

\section{A Lefschetz-type result} \label{section:lefschetz-type}
In this section, which is devoted to the proof of Theorem A, we
use the following notation.

Let $Y \subseteq \mathbb{P}^n$ be a smooth, complete intersection
fourfold and let $X \subseteq Y$ be a reduced and irreducible
threefold which is the intersection of $Y$ with a hypersurface of
  $\mathbb{P}^n$.
  We suppose that the only singularities of $X$ are isolated
  multiple points and we
  denote by $\Sigma=\{p_1, \ldots, p_k\}$ the singular locus of $X$.

We also assume that the tangent cone of $X$ at each point $p_i$ is a
cone over a \emph{smooth} surface of degree $m_i$ in $\mathbb{P}^3$,
and we express this condition by saying that $p_i$ is an
\emph{ordinary multiple point} of multiplicity $m_i$, or \emph{an
ordinary} $m_i$-\emph{ple point}.

Let $\widetilde{Y}:=\mathrm{Bl}_{ \Sigma}(Y)$
  be the blow-up of $Y$ at $\Sigma$, let $\eta \colon \widetilde{Y} \to Y$
  be the blowing-up map, with exceptional divisors $E_1, \ldots, E_k$  and
  write $\widetilde{X}\subset \widetilde{Y}$
   for the strict transform of
  $X$; notice that $\widetilde{X}$ is a smooth threefold.
  Moreover, let $H$  be the pullback on $\widetilde{Y}$ of the
  hyperplane section of $Y$, namely $H=\eta^* \mathcal{O}_Y(1)$.

Finally, we denote by $\pi \colon \widetilde X \to X$ the
restriction
  of $\eta$ to $X$
  and by $\mathcal{E}_i \subset \widetilde X$  the exceptional divisor of
  $\pi$ over
  the point $p_i$, that is $\mathcal{E}_i=\widetilde{X} \cap E_i$. Since
   each $p_i \in X$ is an ordinary
  $m_i$-ple point, $\mathcal{E}_i$  is a smooth surface of
  degree $m_i$ in $\mathbb{P}^3$ and we obtain
\begin{equation} \label{equation:Ei}
  H^1(\mathcal{E}_i,  \, \mathbb{C})=0, \quad H^3(\mathcal{E}_i, \,
  \mathbb{C})=0.
\end{equation}
We can summarize the situation by means of the following
commutative
 diagram.
\begin{equation*} \label{diag:XY}
\xymatrix{\mathcal{E}_i \ar@{^{(}->}[r] \ar@{^{(}->}[d] &
\widetilde{X} \ar[r]^{\pi} \ar@{^{(}->}[d]
& X \ar@{^{(}->}[d]  \\
E_i \ar@{^{(}->}[r] & \widetilde{Y} \ar[r]^{\eta} & Y. }
\end{equation*}

\begin{proposition} \label{proposition:defect}
With the notation above, we have
\begin{equation*} \label{equation:hdefect}
    b_4(\widetilde X)=b_4(X)+k.
\end{equation*}
\end{proposition}
\begin{proof}
Let $\underline{\mathbb{C}}$ be the constant sheaf relative to
$\mathbb{C}$ on $\widetilde X$ and consider the corresponding Leray
spectral sequence for $\pi \colon \widetilde X \to X$, namely
\begin{equation*}
    \mathrm{E}_2^{p,q}= H^p(X, \, R^q \pi_* \underline{\mathbb C}), \quad
    \mathrm{E}_{\infty} \Rightarrow H^{\ast}(\widetilde{X}, \, \mathbb{C}).
  \end{equation*}
  Observe that $R^0\pi_* \underline{ \mathbb C} = \underline{\mathbb{C}}$,
  where by abuse of notation we continue to write
$\underline{\mathbb{C}}$ for the
  constant sheaf relative to $\mathbb{C}$ on $X$. Moreover, since any
  semi-algebraic set has locally a conic structure
  (\cite[Theorem 9.3.6 p.225]{bcr:realag}), it follows  that $X$ is locally
  contractible and
  \begin{equation*}
    R^q \pi_* \underline{\mathbb{C}}= \bigoplus_{i=1}^{k}  H^q(\mathcal{E}_i,
    \,
    \mathbb{C})_{p_i}
  \end{equation*}
  for every $q>0$, where the subscript denotes the skyscraper
  sheaf supported at the point $p_i$. Summing up, we obtain
  \begin{equation} \label{equation:Rp}
    \mathrm{E}_2^{p,q}= \begin{cases}
      H^p(X, \, \mathbb{C}) & p \geq 0, \quad q=0 \\
      \bigoplus_{i=1}^{k} H^q(\mathcal{E}_i, \, \mathbb C) & p=0, \quad q>0 \\
      0 & p>0, \quad q>0.
    \end{cases}
  \end{equation}
  The relevant part of $\mathrm{E}^{p,q}_2$ is shown in Table
  \ref{table:spectralsequence}.

\begin{table}[H] \label{table:spectralsequence}
\begin{tikzpicture}[cells={anchor=base}]
  \matrix (m) [matrix of math nodes, column sep=2mm,row sep=2mm]{
       \quad  q  &           \quad  \strut              &[5mm]\quad \strut  &
       \quad \strut& & & &   &
      \strut \\
       4 & \bigoplus_{i=1}^{k} H^4(\mathcal{E}_i)   & 0 &  \dots & \dots &\dots & \dots & 0 & \strut\\
       3 & \bigoplus_{i=1}^k H^3(\mathcal{E}_i)=0 &  0 & \dots & \dots &\dots & \dots & 0 & \strut\\
       2 & \bigoplus_{i=1}^k H^2(\mathcal{E}_i)   & 0 & \dots & \dots &\dots & \dots & 0 & \strut \\
       1 & \node {\bigoplus_{i=1}^k H^1(\mathcal{E}_i)=0} ; & 0 & \dots & \dots &\dots & \dots & 0 &
     \strut \\
     0 & \node {H^0(X)} ; & H^1(X) & H^2(X) & H^3(X) & H^4(X) & H^5(X) & H^6(X)
     &\quad \strut \\
      \quad \strut& 0 & 1 & 2 & 3 & 4 & 5 & 6 & p \strut\\};
        \draw[->,thick] (m-7-1.north) -- (m-7-9.north) ;
        \draw[thick] (m-6-1.north) -- (m-6-9.north) ;
    \draw[thick] (m-5-1.north) -- (m-5-9.north) ;
    \draw[thick] (m-4-1.north) -- (m-4-9.north) ;
    \draw[thick] (m-3-1.north) -- (m-3-9.north) ;
    \draw[thick] (m-2-1.north) -- (m-2-9.north) ;

    \draw[thick] (-3,-2.5) -- (-3,2.5);
    \draw[thick] (-1.3,-2.5) -- (-1.3,2.5);
    \draw[thick] (0.2,-2.5) -- (0.2,2.5);
    \draw[thick] (1.8,-2.5) -- (1.8,2.5);
    \draw[thick] (3.5,-2.5) -- (3.5,2.5);
    \draw[thick] (5.0,-2.5) -- (5.0,2.5);
    \draw[thick] (6.6,-2.5) -- (6.6,2.5);
    \draw[<-,thick] (m-1-1.east) -- (m-7-1.east) ;
\end{tikzpicture}
\caption{The table $\mathrm{E}^{p,q}_2$.}
\end{table}

Computing the differentials, it is not difficult to check that
$\mathrm{E}_6^{p, q}=\mathrm{E}_{\infty}^{p, q}$, so there is a
direct sum decomposition
\begin{equation} \label{equation:direct-sum}
H^4(\widetilde X, \, \mathbb{C})=\mathrm{E}_6^{4,0} \oplus
\mathrm{E}_6^{3,1} \oplus \mathrm{E}_{6}^{2,2}\oplus
\mathrm{E}_6^{1,3} \oplus \mathrm{E}_6^{0,4}.
\end{equation}
Using \eqref{equation:Ei} and \eqref{equation:Rp} we obtain
\begin{equation*} \label{equation:summands Leray}
\begin{split}
\mathrm{E}_6^{4,0}&=H^4(X, \, \mathbb{C}), \\
\mathrm{E}_6^{3,1}&=\mathrm{E}_6^{2,2}=\mathrm{E}_6^{1,3}=0, \\
\mathrm{E}_6^{0,4}&= \ker \{d_5 \colon \bigoplus_{i=1}^k
H^4(\mathcal{E}_i, \, \mathbb{C}) \to H^5(X, \, \mathbb{C})\}.
\end{split}
\end{equation*}
Moreover $X$ is a complete intersection threefold with only isolated
singularities, hence Theorem \ref{theorem:dimca} yields $H^5(X, \,
\mathbb{C}) \cong H^5(\mathbb{P}^n, \, \mathbb{C})=0$. Then
 \eqref{equation:direct-sum} becomes
\begin{equation*}
H^4(\widetilde X, \, \mathbb{C})= H^4(X, \mathbb{C}) \oplus
\bigoplus_{i=1}^k H^4(\mathcal{E}_i, \, \mathbb C)  .
\end{equation*}
Since each $\mathcal{E}_i$ is a smooth surface, by
Poincar$\acute{\textrm{e}}$ duality we deduce $H^4(\mathcal{E}_i, \,
\mathbb{C}) \cong \mathbb{C}$ and this completes the proof.
\end{proof}

\begin{proposition} \label{proposition:factorial}
With the notation above, if $b_4(X)=1$ then $X$ is factorial.
\end{proposition}
\begin{proof}
 Assume $b_4(X)=1$ and let $S\subset X$ be any reduced, irreducible surface; we must show that $S$
  is a complete intersection in $X$.
  Let $S^\prime$, $X^\prime$
  be general hyperplane sections of $S$ and $X$ respectively, and
  let  $H_X \in \vert \mathcal{O}_X(1) \vert$,  $H_{X^\prime} \in \vert
  \mathcal{O}_{X^\prime}(1) \vert$. By assumption it
  follows that there exist integers $p, \, q$ such that
  \begin{equation*}
    pS\equiv_{\textrm{hom}} qH_X
  \end{equation*}
  on $X$, hence
  \begin{equation*}
    pS^\prime \equiv_{\textrm{hom}} qH_{X^\prime}
  \end{equation*}
  on $X^\prime$. Since $X^\prime$ is a smooth complete intersection surface,
  by Theorem
\ref{teo:lefschetz} the map
  \begin{equation*}
    H^2(\mathbb{P}^{n-1}, \, \mathbb{Z})\to H^2(
    X^\prime, \,
    \mathbb{Z})
  \end{equation*}
  is injective with torsion free cokernel, hence there exists an integer
  $r$ such that
  \begin{equation} \label{eq:homSH}
    S^\prime \equiv_{\textrm{hom}} rH_{X^\prime}
  \end{equation}
on $X^\prime$. Again by Theorem \ref{teo:lefschetz}, we deduce
$H^1(X^ \prime, \, \mathbb{Z})=0$, hence $H^1(X^ \prime, \,
\mathcal{O}_{X^ \prime})=0$. Therefore by looking at the exponential
sequence
\begin{equation*}
0 \to \mathbb{Z} \to \mathcal{O}_{X^\prime} \to \mathcal{O}_{X^
\prime}^* \to 0
\end{equation*}
we see that there is an injective map
\begin{equation*}
\textrm{Pic}\,X^ \prime=H^1(X^ \prime, \, \mathcal{O}_{X^ \prime}^*)
\hookrightarrow H^2(X^ \prime, \, \mathbb{Z}),
\end{equation*}
so \eqref{eq:homSH} implies
  \begin{equation} \label{lin:homSH}
    S^\prime\equiv_{\textrm{lin}} rH_{X^\prime}
  \end{equation}
  on $X^\prime$. Since
  any smooth complete intersection is projectively
  normal (\cite[ex. 8.4 (b) p. 188]{Ha77}), it follows that $S^\prime$
  is the complete intersection of $X^\prime$ with a hypersurface of
  $ \mathbb{P}^{n-1}$ of degree $r$, say $F^\prime$. Then the Koszul
   resolution of $\mathcal{I}_{S^\prime/
    \mathbb{P}^{n-1}}$ shows that
\begin{equation} \label{equation:h1(i)=0}
    H^1( \mathbb{P}^{n-1}, \, \mathcal{I}_{S^\prime/
    \mathbb{P}^{n-1}}(i))=0 \quad \textrm{for all } i \in \mathbb{Z}.
  \end{equation}
Applying the Snake Lemma (\cite[Chapter 2]{GM99}) to the diagram
\begin{equation*}
\xymatrix{ 0 \ar[r] & \mathcal{O}_{\mathbb{P}^n}(i-1) \ar[r] \ar[d]
& \mathcal{O}_{\mathbb{P}^n}(i) \ar[r] \ar[d] &
\mathcal{O}_{\mathbb{P}^{n-1}}(i) \ar[r] \ar[d] & 0 \\
0 \ar[r] & \mathcal{O}_S(i-1) \ar[r] & \mathcal{O}_S(i) \ar[r] &
\mathcal{O}_{S^\prime}(i) \ar[r] & 0}
\end{equation*}
we obtain the short exact sequence
\begin{equation*}
    \xymatrix{ 0 \ar[r]& \mathcal{I}_{S/ \mathbb{P}^n}(i-1) \ar[r]&
    \mathcal{I}_{S/ \mathbb{P}^n}(i) \ar[r] & \mathcal{I}_{S^\prime/
      \mathbb{P}^{n-1}}(i) \ar[r] & 0 }
\end{equation*}
which in turn gives, passing to cohomology,
  \begin{equation}   \label{equation:exsequence2}
    \xymatrix{  H^1( \mathbb{P}^n, \, \mathcal{I}_{S/
      \mathbb{P}^n}(i-1)) \ar[r] & H^1( \mathbb{P}^n, \, \mathcal{I}_{S/
    \mathbb{P}^n}(i) ) \ar[r]& H^1( \mathbb{P}^{n-1}, \, \mathcal{I}_{S^\prime/
      \mathbb{P}^{n-1}}(i) ).
    }
\end{equation}
  Since  $ H^1( \mathbb{P}^n, \, \mathcal{I}_{ S/ \mathbb{P}^n
  })=0$, by using \eqref{equation:h1(i)=0} and \eqref{equation:exsequence2}
  we find by induction that $H^1( \mathbb{P}^n, \, \mathcal{I}_{ S/
  \mathbb{P}^n}(i))=0$ for any $i\ge 0$. In particular the map
  \begin{equation*}
H^0(\mathbb{P}^n, \, \mathcal{I}_{S/
    \mathbb{P}^n}(r)) \to H^0(\mathbb{P}^{n-1}, \, \mathcal{I}_{S^\prime/
    \mathbb{P}^{n-1}}(r))
\end{equation*}
is surjective, so we can lift the hypersurface $F^\prime \in
H^0(\mathbb{P}^{n-1}, \, \mathcal{I}_{S^\prime/
\mathbb{P}^{n-1}}(r))$ to a hypersurface $F \in H^0(\mathbb{P}^{n},
\, \mathcal{I}_{S/ \mathbb{P}^{n}}(r))$. Moreover, such a  $F$ does
not contain $X$ (since $F^\prime$ does not contain $X^\prime$).
Hence it follows by degree reasons that $S$ is the complete
intersection of $X$ with $F$.
\end{proof}

\begin{remark} If $X \subset \mathbb{P}^4$ is a hypersurface of degree $d$ whose
only singularities are ordinary double points, then the converse of
Proposition \ref{proposition:factorial} also holds, namely $X$ is
factorial if and only if $b_4(X)=1$. In fact in the nodal case one
has $b_4(X)=1+ \delta$, where $\delta$ is the \emph{defect} of $X$,
namely the number of dependent conditions imposed by the reduced
subscheme $\Sigma$ to the homogeneous forms of degree $2d-5$. In
other words,
\begin{equation*}
\delta=k+ h^0(\mathbb{P}^4, \, \mathcal{I}_{\Sigma}(2d-5))
-h^0(\mathbb{P}^4, \, \mathcal{O}_{\mathbb{P}^4}(2d-5))=
h^1(\mathbb{P}^4, \, \mathcal{I}_{\Sigma}(2d-5))
\end{equation*}
and one has $\delta=0$ precisely when $X$ is factorial, see
\cite[Chapter 6]{dimca:sing}, \cite{Cy01}, \cite{Ch10}.

If the singular points of $X$ have higher multiplicity, then the
converse of Proposition \ref{proposition:factorial} is in general
no longer true. For instance, let $X \subset \mathbb{P}^4$ be a
cone over a surface $V \subset \mathbb{P}^3$ of degree $d \geq 4$
with $\textrm{Pic}\,V=\mathbb{Z}$. Then $X$ is factorial (Example
\ref{example:cone-anlocfac}) but \cite[formula (4.18)
p.169]{dimca:sing} shows that
\begin{equation*}
b_4(X)=b_2(V) = d^3-4d^2+6d-2 > 1.
\end{equation*}
\end{remark}

We are now ready to prove our Lefschetz-type result, namely Theorem
A of the Introduction.
\begin{theorem} \label{theorem:principal}
Let $Y \subset \mathbb{P}^n$ be a smooth, complete intersection
fourfold and let $X \subset Y$ be a reduced and irreducible
threefold which is the intersection of $Y$ with a hypersurface of
  $\mathbb{P}^n$.
  Suppose that the singular locus $\Sigma=\{p_1, \ldots, p_k \}$ of  $X$
  consists only of
   ordinary multiple points and denote by $\widetilde{X}\subset
\widetilde{Y}$ the
  strict transform of
  $X$ in the blowing-up $\widetilde{Y}:=\mathrm{Bl}_{\Sigma}Y$ of
  $Y$ at $\Sigma$.
If $\widetilde{X}$ is ample in $\widetilde{Y}$, then $X$ is
factorial.
\end{theorem}
\begin{proof}
By Theorem \ref{teo:lefschetz} we have $h^2(Y, \,
\mathbb{C})=h^2(\mathbb{P}^n, \, \mathbb{C})=1$, so after blowing
up the $k$ points in $\Sigma$ we obtain
\begin{equation*}
h^2(\widetilde{Y},\, \mathbb{C})=h^2(Y, \, \mathbb{C})+k=1+k.
\end{equation*}
By assumption $\widetilde{X}$ is ample in $\widetilde{Y}$, so again
by Theorem \ref{teo:lefschetz} we can write
\begin{equation} \label{eq.Xtilde}
h^2(\widetilde{X}, \, \mathbb{C})=h^2(\widetilde{Y},\,
\mathbb{C})=1+k.
\end{equation}
Using Proposition \ref{proposition:defect},
Poincar$\acute{\textrm{e}}$ duality and \eqref{eq.Xtilde} we get
\begin{equation*}
b_4(X)+k=b_4(\widetilde{X})=h^2(\widetilde{X}, \, \mathbb{C})=1+k,
\end{equation*}
hence $b_4(X)=1$ and $X$ is factorial by Proposition
\ref{proposition:factorial}.
\end{proof}

\begin{corollary} \label{cor:princ-Lefschetz}
Same notation as before. If $\widetilde{X}$ is ample in
$\widetilde{Y}$, then the restriction maps $\emph{Cl}
\,\mathbb{P}^n \to \emph{Cl}\,X$ and $\emph{Cl}\,Y \to
\emph{Cl}\,X$ are both isomorphisms.
\end{corollary}
\begin{proof}
The map $\textrm{Cl}\,\mathbb{P}^n \to \textrm{Cl}\,X$ is an
isomorphism by Theorem \ref{theorem:principal}, whereas the map
$\textrm{Cl} \,\mathbb{P}^n \to \textrm{Cl} \,Y$ is an isomorphism
by Theorem \ref{theorem:gr-lef}. Hence $\textrm{Cl}\, Y \to
\textrm{Cl}\,X $ is an isomorphism as well.
\end{proof}
Theorem \ref{theorem:gr-lef-gen} and Corollary
\ref{cor:princ-Lefschetz} imply that, if $\widetilde X$ is ample in
$\widetilde Y$, there is a commutative diagram
\begin{equation}
\xymatrix{
\textrm{Cl}\,Y \ar[r] \ar[d]_{\eta^*} & \textrm{Cl}\,X \ar[d]^{\pi^*} \\
 \textrm{Cl}\,\widetilde{Y} \ar[r] &
 \textrm{Cl}\,\widetilde{X}}
\end{equation}
whose horizontal arrows are both isomorphisms. Moreover the
pull-back map $\pi^* \colon \textrm{Cl}\,X \to
\textrm{Cl}\,\widetilde{X}$ (which can be defined because
$\textrm{Cl}\,X = \textrm{Pic}\, X$) is injective.

\begin{remark} \label{rem:theA}
The converse of Theorem \ref{theorem:principal} is in general not
true. In fact, let $V \subset \mathbb{P}^3$ be a smooth surface of
degree $d \geq 4$ such that $\textrm{Pic}\,V = \mathbb{Z}$ and let
$X \subset \mathbb{P}^4$ be the cone over $V$. Then the divisor
$\widetilde{X} \subset \widetilde{\mathbb{P}}^4$ belongs to the
linear system $|d(H-E)|$, hence $\widetilde{X}^4=0$ and by
Nakai-Moishezon criterion $\widetilde X$ is not ample. However, $X$
is factorial (see Example \ref{example:cone-anlocfac}).
\end{remark}

\section{Applications} \label{section:applications}

Let us give now some applications of the previous results. We
start by showing that if a threefold hypersurface in a good
ambient space has ``few" singularities, which are all ordinary
points, then $X$ is factorial.

\begin{theorem} \label{theorem:few-points}
Let $Y \subset \mathbb{P}^n$ be a smooth, complete intersection
fourfold and $X \subset Y$  be a reduced, irreducible threefold,
which is complete intersection of $Y$ with a hypersurface of
degree $d$. Assume that the singular locus of $X$ consists
precisely of $k$ ordinary multiple points $p_1, \ldots, p_k$ of
multiplicity $m_1, \ldots, m_k$. If
\begin{equation} \label{eq:linear-bound-1-gen}
\sum_{i=1}^k m_i < d
\end{equation}
then $X$ is factorial.
\end{theorem}

\begin{proof}
We use the same notation as in Section \ref{section:lefschetz-type}.
By Theorem \ref{theorem:principal}, it is sufficient to show that
$\widetilde{X}$ is an ample divisor in $\widetilde{Y}$. Since each
$p_i \in X$ is an ordinary singular point, $\widetilde{X}$ is smooth
and we have
\begin{equation} \label{eq:ThmB}
\widetilde{X}=dH-\sum_{i=1}^k m_iE_i=\bigg(d-\sum_{i=1}^km_i
\bigg)H+ \sum_{i=1}^k m_i(H-E_i).
\end{equation}
By \eqref{eq:linear-bound-1-gen}, the linear system
$|\big(d-\sum_{i=1}^km_i \big)H|$ is base point free; since the same
is clearly true for  $|m_i(H-E_i)|$, equation \eqref{eq:ThmB} shows
that $\mathcal{O}_{\widetilde{Y}}(\widetilde{X})$ is a globally
generated line bundle on $\widetilde{Y}$. By \cite[Corollary 1.2.15
p. 28]{lazarsfeld:p1} it remains only to  prove that $\widetilde{X}$
has positive intersection with every effective, irreducible curve
$\widetilde{C} \subset \widetilde{Y}$. If $\widetilde{C} \subset
E_i$ for some $i$ this is clear. Otherwise $C:= \eta_*
\widetilde{C}$ is an effective and irreducible curve on $Y$ and by
the projection formula one has
\begin{equation*}
H \cdot \widetilde{C} = \eta^* \mathcal{O}_Y (1) \cdot \widetilde{C}
= \mathcal{O}_Y (1) \cdot \eta_* \widetilde{C}= \mathcal{O}_Y (1)
\cdot C >0.
\end{equation*}
On the other hand, $(H-E_i) \cdot \widetilde{C} \geq 0$ because
$|H-E_i|$ is base-point free. Then \eqref{eq:ThmB}
 implies  $\widetilde{X} \cdot \widetilde{C}
>0$ and we are done.
\end{proof}

As far as we know, Theorem \ref{theorem:few-points} provides the
first factoriality criterion for complete intersection threefolds
in $\mathbb{P}^n$ with ordinary singularities. For \emph{nodal}
threefolds in $\mathbb{P}^4$ and $\mathbb{P}^5$ some sharper
results were previously obtained in \cite{Ch10} and
\cite{kosta:ci}, respectively. See also \cite{Sab04} and
\cite{cd:ci}. When $X \subset \mathbb{P}^4$ and $k=1$, namely when
we have exactly one (ordinary) singular point, Theorem
\ref{theorem:few-points} can be deduced from \cite[Theorem 4.17 p.
214]{dimca:sing}. Since we find this other proof of independent
interest, for the sake of completeness we include it in the
Appendix.

\begin{example} \label{example:unique-sing}
Let $m, \, d \in \mathbb{N}$ with $m < d$ and take a homogeneous
polynomial $f_m(x_0, \, x_1, \, x_2, \, x_3)$
  of degree $m$
  such that $V:=V(f_m)\subset \mathbb{P}^3$ is a smooth surface.
  Given general forms $f_{m+1}, \, f_{m+2}, \ldots, f_d \in
  \mathbb{C}[x_0, \, x_1, \,x_2, \,x_3]$, of respective degrees $m+1, \,
  m+2,  \ldots, d$, the polynomial
  \begin{equation*}
    f:= x_4^{d-m}f_m+x_4^{d-m-1}f_{m+1}+\cdots + f_d
  \end{equation*}
  defines a hypersurface $X \subset \mathbb{P}^4$ of degree $d$ with
  a unique singular point, namely
  $p=[0:0:0:0:1]$, which is ordinary of multiplicity $m$.
  Then $X$ is factorial by Theorem \ref{theorem:few-points}.
\end{example}

\begin{remark} \label{rem:kollar}
In the statement of Theorem \ref{theorem:few-points}, the
condition that all the singularities are ordinary is an essential
one, as shown by the following example due to Koll\'{a}r, see
\cite[p. 108]{Mel04}. Consider a general quartic hypersurface $X
\subset \mathbb{P}^4$ whose
 defining polynomial is in the span of the monomials $\{x_0^4, \, x_1^4, \,
(x_4^2x_3+x_2^3)x_0, \, x_3^3x_1, \, x_4^2x_1^2 \}$. Then $X$ has
the unique singular point $p=[0: \,0: \,0: \,0: \,1]$, but it is
not factorial because it contains the plane $\{x_0=x_1=0\}$.
Notice that $p \in X$ is a non-ordinary double point, because the
corresponding tangent cone is a cone over a singular quadric
surface in $\mathbb{P}^3$ (in fact, $p$ is a so-called
$cA_1$-singularity, see \cite{Re87}).
\end{remark}

If the bound \eqref{eq:linear-bound-1-gen} is not satisfied, we
can still give a factoriality criterion for a hypersurface $X
\subset \mathbb{P}^4$ provided that its singularities are in
general position and they all have the same multiplicity.

\begin{theorem} \label{theorem:worstthannodes}
  Let $\Sigma:=\{p_1,\ldots,p_k \}$ be a set of $k$ distinct,
  general points in $\mathbb{P}^4$ and let $d, \, m$ be
  positive integers with $d \geq m$.
\begin{itemize}
\item[$\boldsymbol{(i)}$] If
  \begin{equation} \label{equation:nonempty}
    \bigg \lfloor \frac{d+5}{m+4} \bigg \rfloor^4 > k
\end{equation}
  then there exists a hypersuface $X \subset \mathbb{P}^4$ of degree $d$,
  with
  $k$ ordinary $m$-ple points at $p_1,\ldots,p_k$ and no other
  singularities. \\
\item[$\boldsymbol{(ii)}$]  If the stronger condition
  \begin{equation}  \label{equation:numample}
   \min \bigg\{ \bigg \lfloor \frac{d+5}{m+4}  \bigg \rfloor^4, \, \bigg
   \lfloor \frac{d}{m}
   \bigg \rfloor^4 \bigg\} > k
\end{equation}
  holds, then any hypersurface $X$ as in part $\boldsymbol{(i)}$
  is factorial.
\end{itemize}
\end{theorem}
\begin{proof}
We set $Y=\mathbb{P}^4$ and we use the same notation as in Section
\ref{section:lefschetz-type}. Since $d \geq m$ we have
\begin{equation*}
\bigg \lfloor \frac{d+4}{m+3} \bigg \rfloor^4
 > \bigg \lfloor    \frac{d+5}{m+4} \bigg \rfloor^4>k,
   \end{equation*}
so by Corollary \ref{corollary:amplenessbound} the two divisors
  \begin{equation*}
  \begin{split}
  (d+4)H-(m+3)\sum_{i=1}^k E_i \quad \textrm{and} \quad
  (d+5)H-(m+4)\sum_{i=1}^k E_i
\end{split}
  \end{equation*}
are ample on $\widetilde{Y}$ and by Kodaira vanishing theorem we
deduce
  \begin{equation} \label{eq:kodaira}
    H^1 \bigg(\widetilde{Y}, \, (d-1)H-m \sum_{i=1}^k E_i \bigg)=0 \quad
    \textrm{and} \quad
    H^1 \bigg(\widetilde{Y}, \, dH-(m+1) \sum_{i=1}^k
    E_i \bigg)=0.
  \end{equation}
By using the two exact sequences
\begin{equation*}
0 \to \mathcal{O}_{\widetilde{Y}} \bigg((d-1)H-m \sum_{i=1}^k E_i
\bigg) \to \mathcal{O}_{\widetilde{Y}} \bigg(dH-m \sum_{i=1}^k E_i
\bigg) \to \mathcal{O}_H \bigg(dH-m \sum_{i=1}^k E_i \bigg) \to 0
\end{equation*}
\begin{equation*}
0 \to \mathcal{O}_{\widetilde{Y}} \bigg(dH-(m+1) \sum_{i=1}^k E_i
\bigg) \to \mathcal{O}_{\widetilde{Y}} \bigg(dH-m \sum_{i=1}^k E_i
\bigg) \to \mathcal{O}_{\sum_{i=1}^k E_i} \bigg(dH-m \sum_{i=1}^k
E_i \bigg) \to 0
\end{equation*}
and \eqref{eq:kodaira}, we now see that the restriction maps
\begin{equation} \label{equation:complete.H}
H^0 \bigg(\widetilde{Y}, \,dH-m \sum_{i=1}^k E_i \bigg) \to H^0
\bigg(H, \, dH-m \sum_{i=1}^k E_i \bigg)
\end{equation}
\begin{equation} \label{equation:complete.Ei}
H^0 \bigg(\widetilde{Y}, \,dH-m \sum_{i=1}^k E_i \bigg) \to H^0
\bigg(E_j, \, dH-m \sum_{i=1}^k E_i \bigg), \quad j=1, \ldots, k
\end{equation}
are all surjective. This means that the linear system $\vert dH- m
\sum_{i=1}^k  E_i  \vert $ restricts to a complete linear
  system on the general element of $\vert H
  \vert$ and on each $E_j$; therefore $\vert dH- m \sum_{i=1}^k  E_i  \vert
  $ is base-point free and by Bertini's theorem its
  general element $\widetilde{X}$ is smooth and irreducible.
Then $X:=\eta_*(\widetilde{X})$ is the desired hypersurface, and
this proves $\boldsymbol{(i)}$.

Finally, if \eqref{equation:numample} holds then $\widetilde{X}$
is ample in $\widetilde{Y}$
  by Corollary \ref{corollary:amplenessbound}, hence $X$ is factorial
by Theorem \ref{theorem:principal}. This proves
$\boldsymbol{(ii)}$.
\end{proof}

\begin{corollary} \label{corollary:worstthannodes}
If $d \geq \frac{5}{4}m$ and \emph{\eqref{equation:nonempty}}
holds, than \emph{any} hypersurface as in Theorem
\emph{\ref{theorem:worstthannodes}}, part $\boldsymbol{(i)}$ is
factorial.
\end{corollary}
\begin{proof}
The assumptions imply
\begin{equation*}
\bigg \lfloor \frac{d}{m} \bigg \rfloor^4  \geq \bigg \lfloor
\frac{d+5}{m+4} \bigg \rfloor^4 > k,
\end{equation*}
so the claim follows by Theorem \ref{theorem:worstthannodes}, part
$\boldsymbol{(ii)}$.
\end{proof}

The following examples show that the numerical inequalities in
Theorem \ref{theorem:worstthannodes} are not sharp.
\begin{itemize}
\item Let $X \subset \mathbb{P}^4$ be a hypersurface of degree $3$
with two ordinary double points and no other singularities. Then
\eqref{equation:numample} is not satisfied, but $X$ is factorial
(\cite{Ch10}).
\item Let $V \subset \mathbb{P}^3$ be a smooth
surface of degree $d \geq 4$ such that $\textrm{Pic}\,V =
\mathbb{Z}$ and let $X \subset \mathbb{P}^4$ be the cone over $V$.
Then \eqref{equation:numample} is not satisfied, but $X$ is
factorial (Example \ref{example:cone-anlocfac}).
\item Let $X
\subset \mathbb{P}^4$ be a hypersurface of degree $d$ with a
unique singularity, which is ordinary of multiplicity $m$. If $m<
d < 2m+3$ then \eqref{equation:numample} is not satisfied, but $X$
is factorial (Theorem \ref{theorem:few-points}).
\end{itemize}

\section{Non-factorial examples} \label{section:non-factorial}

This section is devoted to the construction of some examples of
non-factorial
 hypersurfaces in $\mathbb{P}^4$ with only ordinary
multiple points as singularities. They generalize the examples of
non-factorial, nodal hypersurfaces  described in Example
\ref{example:fact-2}. More precisely, we prove the following result.

\begin{proposition} \label{prop:non-factorial}
For any pair $(t, \, \delta)$ of positive integers there exists a
non-factorial hypersurface $X \subset \mathbb{P}^4$ of degree $d$
with $k$ ordinary $m$-ple points as only singularities, where
\begin{equation*}
d=\delta t+1, \quad k=\delta^2, \quad m=t+1.
\end{equation*}
\end{proposition}
\begin{proof}
Consider a general pencil of curves of degree $\delta$ in the
projective plane $\mathbb{P}^2$ with homogeneous coordinates
$[x_2: \, x_3 \, : x_4]$. Let  $F_1, \, F_2, \ldots F_t$ and $G_1,
\, G_2, \ldots G_t$ be general elements in the pencil; then the
products of the $F_i$ defines a plane curve of degree $\delta t$
with $\delta^2$ points of multiplicity $t$ at the base points
$[a_k: \, b_k: \, c_k]$ of the pencil, and similarly for the
product of the $G_i$. Next, let us define
\begin{equation} \label{eq:F-G}
\begin{split}
F&=\prod_{i=1}^t F_i(x_2, \, x_3, \, x_4) + \sum_{j=t}^{\delta t}
\bigg(\sum_{\alpha+\beta+\gamma=\delta t-j} \Phi^j_{\alpha \beta
\gamma}(x_0, \, x_1)\, x_2^{\alpha} x_3^{\beta} x_4^{\gamma}
\bigg), \\
G&=\prod_{i=1}^t G_i(x_2, \, x_3, \, x_4) + \sum_{j=t}^{\delta t}
\bigg(\sum_{\alpha+\beta+\gamma=\delta t-j} \Psi^j_{\alpha \beta
\gamma}(x_0, \, x_1)\, x_2^{\alpha} x_3^{\beta} x_4^{\gamma}\bigg),
\end{split}
\end{equation}
where $\Phi^j_{\alpha \beta \gamma}(x_0, \,x_1)$, $\Psi^j_{\alpha
\beta \gamma}(x_0, \,x_1)$ are general homogeneous forms of degree
$j$ in the variables $x_0$, $x_1$. Now, set
\begin{equation} \label{eq:x0x1FG}
f=x_0F+x_1G.
\end{equation}
We claim that the hypersurface $X=V(f) \subset \mathbb{P}^4$
(whose degree is $d=\delta t+1$) has the desired properties.
Indeed, varying $F_i$, $G_i$, $\Phi^j_{\alpha \beta \gamma}$,
$\Psi^j_{\alpha \beta \gamma}$, the polynomials $f$ define a
linear system contained in $|\mathcal{O}_{\mathbb{P}^4}(d)|$,
whose base locus is the plane $\pi$ of equations $x_0=x_1=0$.
Thus, using Bertini's theorem one easily deduces that the only
singular points of the general hypersurface $X$ constructed in
this way are the $\delta^2$ points $[0:\, 0: \, a_k: \, b_k: \,
c_k] \in \pi$. Moreover $X$ is obviously not factorial, since $\pi
\subset X$.

It remains only to show that each $p_i \in X$ is an ordinary
$(t+1)$-ple point. Up to a linear change of coordinates involving
only $x_2$, $x_3$ and $x_4$ we may assume  $p_i=[0:\, 0: \, 1: \, 0:
\, 0]$. Write
\begin{equation*}
\begin{split}
\prod_{i=1}^t F_i(1, \, x_3, \, x_4) &= \prod_{i=1}^t (u_{3, i} \,
x_3
+ u_{4,i}\, x_4)+ \textrm{higher order terms}, \\
 \prod_{i=1}^t G_i(1, \, x_3, \, x_4) &= \prod_{i=1}^t (v_{3,
i}\, x_3 + v_{4,i}\, x_4)+ \textrm{higher order terms},
\end{split}
\end{equation*}
where $u_{3, i}, \, u_{4, i}, v_{3, i}, \, v_{4, i} \in \mathbb{C}$.
Since the variable $x_2$ appears in the expression of the second summands
of $F$ and $G$
with exponent at most $t(\delta-1)$, the equation of the tangent
cone of $X$ at $p_i$ is given by
\begin{equation*}
 x_0 \bigg(\prod_{i=1}^t (u_{3, i}\, x_3 + u_{4,i}\, x_4) +
 \Phi^t_{t(\delta-1) \, 0 \,0} (x_0, \, x_1)\bigg) + x_1
 \bigg(\prod_{i=1}^t (v_{3, i} \, x_3 + v_{4,i}\, x_4) +
 \Psi^t_{t(\delta-1) \, 0 \,0}(x_0, \, x_1)\bigg)=0
\end{equation*}
and for a general choice of the parameters this defines a cone
over a smooth surface of degree $t+1$ in $\mathbb{P}^3$. Then the
proof is complete.
\end{proof}

Observe that all the examples in Proposition
\ref{prop:non-factorial} satisfy
\begin{equation*}
k(m-1)^2=(d-1)^2.
\end{equation*}
On the other hand, in \cite{Sab04} it is proven that if the
singular locus of $X $ consists of $k_2$ ordinary double points
and $k_3$ ordinary triple points and if $k_2 + 4 k_3 < (d -1)^2$,
then any smooth surface contained in $X$ is a complete
intersection in $X$. Motivated by this result, we make the
following conjecture, which generalizes the theorem of Ciliberto,
Di Gennaro and Cheltsov stated in the Introduction.

\begin{conjecture} \label{conjecture:factoriality}
Let $X \subset \mathbb{P}^4$ be a hypersurface of degree $d$, whose
singular locus consists of $k$ ordinary multiple points $p_1,
\ldots, p_k$ of multiplicity $m_1, \ldots, m_k$. If
\begin{equation} \label{eq:ineq-fact}
\sum_{i=1}^k (m_i-1)^2 < (d-1)^2
\end{equation}
then $X$ is factorial.
\end{conjecture}
Theorem \ref{theorem:few-points} shows that Conjecture
\ref{conjecture:factoriality} is true for $k=1$.

\section{Appendix}

In the present Appendix we give a different proof of Theorem
\ref{theorem:few-points} in the case $k=1$ (i.e. for hypersurfaces
with only one singularity) by applying some results from
\cite{dimca:sing}. More precisely, we make use of the following

\begin{lemma} \label{lemma:mudet}
    Let $(X, \,0)$ be a germ of an affine hypersurface in $ \mathbb{C}^4$,
    defined by the polynomial $f$,
    such that $X$ has an isolated singularity of multiplicity $m$ at
    $0$. If the
    singularity is ordinary, then
    \begin{equation*}
      \mu \textrm{-}\det(X,0)=m,
    \end{equation*}
where $\mu \textrm{-} \det(X,\,0)$ denotes the smallest positive
integer $s$ such that the family $f_t=f+th$, $t \in [0,
\,\varepsilon)$ is $\mu$-constant for any germ $h$ vanishing of
order $s$ at $0$ and $\varepsilon >0$ small enough  $($here $\mu$
is the Milnor number and $\varepsilon$ may depend on h$)$.
\end{lemma}
\begin{proof}
If $\deg (f)=d$ we can write
    \begin{equation*}
      f = f_m+f_{m+1}+\ldots + f_d,
    \end{equation*}
    where $f_i$ is the homogeneous piece of degree $i$. Since $0\in X$ is an ordinary singularity, the tangent cone
    $V(f_m) \subset \mathbb{C}^4$  has an isolated singularity at the
    origin; then $f$ defines a so called \emph{semiquasihomogeneous} (SQH) \emph{hypersurface
    singularity} with principal part $f_m$, see \cite[p. 123]{GLS07}. So by \cite[pag.~74]{dimca:sing}
    or \cite[Corollary 2.18]{GLS07} we obtain
\begin{equation*}
    \mu(f) = \mu (f_m)=(m-1)^4.
\end{equation*}
    Let $h$ be a general homogeneous polynomial of degree $s$ with $s <
    m$. For $t \neq 0$ the polynomial $f_t=f + th$ defines a SQH
    hypersurface singularity with principal part $th$, hence
    \begin{equation*}
\mu(f+th)=\mu(th) = (h-1)^4 < (m-1)^4 = \mu(f)
\end{equation*}
and the family $f_t$ is \emph{not} $\mu$-constant. Therefore we have
$ \mu \textrm{-}\det(X,0)\geq m$.

It remains to show that $\mu \textrm{-} \det(X,0)\leq m$. Let $h$
be a germ vanishing of order $m$ at the
    origin; then we want to prove that $\mu(f+th)=\mu(f)$ if $t$ is small enough.
    We have
    \begin{equation*}
      f+th=(f_m+th_m)+g,
    \end{equation*}
    where all the monomials appearing in $g$ have degree at least $m+1$. Moreover,
    if $t$ is small enough, $V(f_m+th_m)$ has an isolated singularity
at the origin.
    Hence $f+th$ defines a SQH hypersurface
    singularity with principal part $f_m+th_m$, so
    \begin{equation*}
      \mu (f+th)= \mu (f_m+t h_m)=(m-1)^4=\mu(f).
   \end{equation*}
This completes the proof.
   \end{proof}

Now let $Y \subset \mathbb{P}^n$ be a smooth, complete intersection
fourfold and $X \subset Y$  be a reduced, irreducible threefold,
which is complete intersection of $Y$ with a hypersurface of degree
$d$. Assume that the singular locus of $X$ consists precisely of one
ordinary multiple points $p$ of multiplicity $m < d$. We want to
show that $X$ is factorial. By Lemma \ref{lemma:mudet}, the
assumption
  $m < d$ becomes
  \begin{equation*}
    m= \mu\textrm{-}\det (X,p)<d,
  \end{equation*}
  then by \cite[Theorem 4.17 p. 214]{dimca:sing} we obtain $\Delta_X=1$, where
  $ \Delta_X$ is the Alexander polynomial of $X$. By \cite[p. 146 and p.~206]{dimca:sing}
  this in turn implies
   $H^4_0(X)=0$,
  where $H_0$ stands for the primitive cohomology, namely
  \begin{equation*}
    H_0^4(X)= \textrm{coker} \{ H^4( \mathbb{P}^4, \, \mathbb{C}) \to H^4(X,
    \, \mathbb{C}) \}.\
  \end{equation*}
  In other words, the restriction map $H^4( \mathbb{P}^4, \, \mathbb{C}) \to H^4(X,
    \, \mathbb{C})$ is surjective, hence $b_4(X) \leq 1$. Since the class of the
  hyperplane section of $X$ is certainly
  nonzero in $H^4(X, \, \mathbb{C})$, it follows
   $b_4(X)=1$, thus $X$ is factorial by Proposition
   \ref{proposition:factorial}.

\def\cprime{$'$}

\bigskip
\bigskip

Francesco Polizzi, \\ Dipartimento di Matematica, Universit\`{a}
della Calabria, Cubo 30B, 87036 \\ Arcavacata di Rende, Cosenza (Italy)\\
\emph{E-mail address:} \verb|polizzi@mat.unical.it| \\ \\

Antonio Rapagnetta, \\ Dipartimento di Matematica, Universit\`{a}
di Roma 2 Tor Vergata \\
Via Della Ricerca Scientifica 1, 00133 Roma (Italy) \\
\emph{E-mail address:} \verb|rapagnet@axp.mat.uniroma2.it| \\ \\

Pietro Sabatino, \\ Dipartimento di Matematica, Universit\`{a}
della Calabria, Cubo 30B, 87036 \\ Arcavacata di Rende, Cosenza (Italy)\\
 \emph{E-mail address:} \verb|pietrsabat@gmail.com|

\end{document}